\documentclass[a4paper,10pt]{amsart}

\usepackage{amssymb, amsmath}

\newcommand{\pD}[2]{\frac{\partial #1}{\partial #2}}

\newcommand{\vn}[1]{\lVert#1\rVert}
\newcommand{\IP}[2]{\left< #1 , #2 \right>}
\newcommand{\eIP}[2]{\left(\left. #1 \hskip+0.5mm\right| #2 \right)}
\newcommand{\rD}[2]{\frac{d #1}{d #2}}

\newcommand{\R}{\ensuremath{\mathbb{R}}}
\newcommand{\N}{\ensuremath{\mathbb{N}}}

\newtheorem{thm}{Theorem}[section]
\newtheorem{cor}[thm]{Corollary}
\newtheorem{prop}[thm]{Proposition}
\newtheorem{lem}[thm]{Lemma}

\title{Lifespan theorem for simple constrained surface diffusion flows}
\author{Glen Wheeler}

\address{School of Mathematics and Applied Statistics\\
         University of Wollongong\\
         Northfields Ave\\
         Wollongong, NSW 2500, Australia}

\begin{document}

\begin{abstract}
We consider closed immersed hypersurfaces in $\R^3$ and $\R^4$ evolving by a special class of
constrained surface diffusion flows.  This class of constrained flows includes the classical surface
diffusion flow.  In this paper we present a Lifespan Theorem for these flows, which gives a positive
lower bound on the time for which a smooth solution exists, and a small upper bound on
the total curvature during this time.  The hypothesis of the theorem is that the surface is not
already singular in terms of concentration of curvature.  This turns out to be a deep property of
the initial manifold, as the lower bound on maximal time obtained depends precisely upon the
concentration of curvature of the initial manifold in $L^2$ for $M^2$ immersed in $\R^3$ and
additionally on the concentration in $L^3$ for $M^3$ immersed in $\R^4$.  This is stronger than a
previous result on a different class of constrained surface diffusion flows, as here we obtain an
improved lower bound on maximal time, a better estimate during this period, and eliminate any
assumption on the area of the evolving hypersurface.
\end{abstract}

\maketitle

\section{Introduction}

Let $f:M^n\times[0,T)\rightarrow\R^{n+1}$ be a family of compact immersed hypersurfaces $f(\cdot,t)
= f_t:M^n\rightarrow f_t(M) = M_t \subset \R^{n+1}$ with associated Laplace-Beltrami operator
$\Delta$, unit normal vector field $\nu$, and mean curvature function $H$.  In this paper we study
the constrained surface diffusion flows, where $f$ evolves by
\begin{equation}
\label{CSD}
\pD{}{t}f = (\Delta H + h)\nu,
\end{equation}
where $h:[0,T)\subset I \rightarrow\R$ is called the \emph{constraint function}.  The study of the
fourth order degenerate parabolic quasilinear system of equations \eqref{CSD} is motivated primarily
by choice of constraint function.  The trivial example of $h=0$, classical surface diffusion flow,
is instructive and for this paper our chief motivator.

Indeed, there does already exist a large body of work on the classical surface diffusion flow.  First
proposed by the physicist Mullins \cite{mullins1957ttg} in 1957
, it was originally designed to model the formation of tiny thermal
grooves in phase interfaces where the contribution due to evaporation-condensation was
insignificant.
Some time later, Davi, Gurtin, Cahn and Taylor \cite{cahn1994sms,davi1990mpi}
proposed many other physical models which give rise to the surface diffusion flow.  These all
exhibit a reduction of free surface energy and conservation of volume; an essential characteristic
of surface diffusion flow.  There are also other motivations for the study of surface diffusion
flow.  For example, two years later Cahn, Elliot and Novick-Cohen \cite{cahn1996che} proved that
the surface diffusion flow is the singular limit of the Cahn-Hilliard equation with a concentration
dependent mobility.  Among other applications, this arises in the modeling of isothermal separation
of compound materials.

Analysis of the surface diffusion flow began slowly, with the first works appearing in the early
80s.  Baras, Duchon and Robert \cite{baras1984edi} showed the global existence of weak solutions for
two dimensional strip-like domains in 1984.  Later, in 1997 Elliot and Garcke \cite{elliott1997erd}
analysed the surface diffusion flow of curves, and obtained local existence and regularity for $C^4$-initial
curves, and global existence for small perturbations of circles.  Significantly, Ito
\cite{Ito1999sdf} showed in 1998 that convexity will not be preserved under the surface diffusion
flow, even for smooth, rotationally symmetric, closed, compact, strictly convex initial
hypersurfaces.  In contrast with the case for second order flows such as mean curvature flow, this
behaviour appears pathological.  Escher, Mayer and Simonett \cite{escher98surface} gave several
numerical schemes for modeling surface diffusion flow, and have also given the only two known numerical
examples \cite{mayer2001nss} of the development of a singularity: a tubular spiral and thin-necked
dumbbell.  They also provide an example of an immersion which will self-intersect under the flow, a
figure eight knot.  In 2001, Simonett \cite{simonett2001wfn} used centre manifold techniques to show
that for initial data $C^{2,\alpha}$-close to a sphere, both the surface diffusion and Willmore
flows (Willmore flow in one codimension is $\partial_t f = \Delta H + \vn{A^o}^2H$, where $A^o = A -
\frac{1}{n}gH$) exist for all time and converge asymptotically to a sphere.

There have been many important works on fourth order flows of a slightly different character, from
Willmore flow of surfaces to Calabi flow, a fourth order flow of metrics.  Significant contributions
to the analysis of these flows by the authors Kuwert, Sch\"atzle, Polden, Huisken, Mantegazza and
Chru\'sciel \cite{chrusciel1991sge,kuwert2001wfs,kuwert2002gfw,mantegazza2002sge,polden:cas} are
particularly relevant, as the methods employed there are similar to ours here.  For the study of
constrained flows, we mention the papers 
\cite{athanassenas1997volume,athanassenas2003behaviour,blokhuis1999helfrich,bohle2008constrained,chung1993elastic,helfrich1990calculation,helfrich1985effect,huang1986deformation,lipowsky1991conformation,J1,J2,J3,MWW11,ros1995stability,willmore2000surfaces},
which contain a plethora of applications to motivate the study of non-trivial constraint functions
$h$.

The issue of local well-posedness of \eqref{CSD} is delicate, although standard, and ovecome with
standard techniques as in \cite{escher98surface}, with the constraint function causing no additional
difficulty.  We make no effort to pose an optimal version, although the interested reader may enjoy
\cite{kochlamm} for recent progress in this direction.

\begin{thm}[Short time existence]
For any smooth initial immersion $f_0:M^n\rightarrow\R^{n+1}$ and bounded constraint function
$h:I\rightarrow\R$, with $I$ an interval containing $0$, there exists a unique nonextendable smooth
solution $f:M\times[0,T)\rightarrow\R^{n+1}$ to \eqref{CSD} with $f(\cdot,0) = f_0$, where $0 < T
\le \infty$.  
\end{thm}

The main issue then becomes global existence.  While we do not treat this explictly here, we do
present a result with applications to singularity analysis, as can be seen in \cite{mySDLTE}.  In
our proof, we exploit the fact that for an $n$-dimensional immersion the integral
\[
\int_M \vn{A}^n d\mu
\]
is scale invariant.  The technique used by Struwe \cite{struwe1985ehm} is then relevant, although as
with all higher order flows the major difficulty is in overcoming the lack of powerful techniques
unique to the second order case.  In particular, we are without the maximum principle, and this
implies that the geometry of the surface could deteriorate, as in \cite{Ito1999sdf}.  Drawing
inspiration from Kuwert and Sch\"atzle \cite{kuwert2002gfw} in particular, we use local integral
estimates to derive derivative curvature bounds under a local smallness of curvature assumption.  In
calculating these estimates it is crucial to only use inequalities which involve universal
constants.  Interpolation inequalities similar in nature to those used by Ladyzhenskaya, Ural'tseva
and Solonnikov \cite{ladyzhenskaya1968laq} and Hamilton \cite{RH}, and the Sobolev inequality of
Michael--Simon \cite{michael1973sam}, are invaluable in this regard.

Following Hamilton \cite{RH}, we denote polynomials in the iterated covariant derivatives of a
tensor $T$ by
\[
P_j^i(T) = \sum_{k_1+\ldots+k_j = i} c_{k_1\cdots k_j}\nabla_{(k_1)}T*\cdots*\nabla_{(k_j)}T,
\]
where $c_{k_1\cdots k_j} \in \R$ and $\nabla_{(k)}T$ is the $k$-th iterated covariant derivative of
$T$; see Section 2 for more details.
For a large class of constrained surface diffusion flows the following theorem applies.

\begin{thm}[Lifespan Theorem, \cite{MWW11}]
\label{t_lifespan}
Suppose $n\in\{2,3\}$ and let $f:M^n\times[0,T)\rightarrow\R^{n+1}$ be a compact immersion with
$C^\infty$ initial data evolving by \eqref{CSD}.
Suppose that for some $j,k,l\in\N_0$ the constraint function
$h:I\supset[0,T)\rightarrow\R$ obeys an estimate
\begin{equation}
h \le \int_M P_j^2(A) + P_k^1(A) + P_l^0(A) d\mu.
\label{LTnewstathstruct}
\end{equation}
Then there are constants $\rho>0$, $\epsilon_0>0$, and $c<\infty$ such that
\begin{equation}
\label{LTnewstatass}
\int_{f^{-1}(B_\rho(x))} \vn{A}^m d\mu\Big|_{t=0} = \epsilon(x) \le \epsilon_0
\qquad
\text{ for any $x\in\R^{n+1}$}
\end{equation}
where $m = \max\{2k-2,2j-k,l,n^2+n-2\}$; and there exists an absolute constant
$C_{A\!B}\in(0,\infty)$ such that
\begin{equation}
\label{AB}
|M_t| \le C_{A\!B},
\qquad\text{ for }\qquad 0 \le t \le \frac{1}{c}\rho^4;
\end{equation}
then the maximal time $T$ of smooth existence for the flow \eqref{CSD} with initial data $f_0 =
f(\cdot,0)$ satisfies
\begin{equation}
\label{eq2}
T \ge \frac{1}{c}\rho^4,
\end{equation}
and we have the estimate
\begin{equation}
\label{eq3}
\int_{f^{-1}(B_\rho(x))} \vn{A}^n d\mu \le c\epsilon(x)
\qquad\text{ for }\qquad
0\le t \le \frac{1}{c}\rho^4.
\end{equation}
\end{thm}

The result we present here is new for the surface diffusion flow, stronger than Theorem
\ref{t_lifespan}, and plays a key role in the analysis of the asymptotic behaviour of the flow. In
particular the main theorem of this papers enables one to guarantee that under certain conditions
finite time curvature singularities possess properties which combined with the results on blowups in
\cite{mySDLTE} allows one to rule out their development entirely.  The key improvements are that the
assumption on the evolving surface area \eqref{AB} is completely removed, and the concentration of
curvature assumption \eqref{LTnewstatass} is in $L^2$ for two dimensional manifolds and additionally
in $L^3$ for three dimensional manifolds.

The reason for these improvements is that we consider only constraint functions which fit into the
following natural class.
A constraint function $h:[0,T)\subset I\rightarrow\R$ which satisfies an estimate
\begin{equation}
\vn{h}_{\infty,J} \le c_h < \infty
\label{simpledef}
\end{equation}
on any closed interval $J\subset [0,T)$ with $c_h=c_h(J)$ is called \emph{simple}.  Note that this
includes constraint functions which are unbounded on $\R$, change sign, and so on.  The
corresponding constrained surface diffusion flow where the constraint function is simple is called
briefly a \emph{simple constrained surface diffusion flow}.  Our main result in this paper is the
following.

\begin{thm}
\label{ASLTlifespan}
Suppose $n\in\{2,3\}$ and let $f:M^n\times[0,T)\rightarrow\R^{n+1}$ be a simple constrained surface
diffusion flow.
Then there are constants $\rho>0$, $\epsilon_0>0$, and $c<\infty$ such that
\begin{equation}
\label{ASLTlifespansmallness}
\int_{f^{-1}(B_\rho(x))} \vn{A}^m d\mu\Big|_{t=0} = \epsilon(x) \le \epsilon_0
\qquad
\text{ for $m = 2,n$, any $x\in\R^{n+1}$},
\end{equation}
and $h$ is simple on $\big[0,\frac{1}{c}\rho^4]$,
then the maximal time $T$
 satisfies
\begin{equation}
\label{ASLTlifespanTestimate}
T \ge \frac{1}{c}\rho^4,
\end{equation}
and we have the estimate
\begin{equation}
\label{ASLTlifespanCurvestimate}
\int_{f^{-1}(B_\rho(x))} \vn{A}^n+\vn{A}^2 d\mu \le c\epsilon(x)
\qquad\text{ for }\qquad
0\le t \le \frac{1}{c}\rho^4.
\end{equation}
\end{thm}
There is no easy relationship between the geometrically motivated constraint functions considered in
\cite{MWW11} and the simple constraint functions considered here.  Despite the stronger statement
Theorem \ref{ASLTlifespan}, one may consider the class of simple constraint functions as being
`larger' than the class of constraint functions which satisfy the geometric growth condition
\eqref{LTnewstathstruct}.  This is due to the following fact.  In \cite{MWW11} we prove that every
constraint function satisfying the growth condition \eqref{LTnewstathstruct} and giving rise to an
area bound as in \eqref{AB} is in fact bounded, given that the concentration of curvature in a high
enough $L^p$ norm is sufficiently small.  In this sense one may regard those functions as satisfying
\eqref{simpledef} under the additional condition that \eqref{ASLTlifespansmallness} holds for later
times in a higher $L^p$ norm.  Additionally, note that there are constraint functions such as $h(t)
= e^t$, $h(t) = \sin t$, $h(t) = \frac{1}{1+t}$, $h(t) = -t$, which easily satisfy \eqref{simpledef}
but do not fit into the framework of \cite{MWW11}.  These may be of interest to model expanding,
breathing, stabilising and shrinking solutions.  Thus we feel that, given the motivating example of
classical surface diffusion flow, one must take both Theorem \ref{t_lifespan} and Theorem
\ref{ASLTlifespan} into account to form a complete picture. 

\section*{Acknowledgements}

This work forms part of the author's PhD thesis under Dr. James McCoy and Prof. Graham Williams,
supported by an Australian Postgraduate Award and the University of Wollongong.  He is grateful for
their advice and for many helpful discussions on this topic.

\section{Notation and preliminary results}

In this section we will collect various general formulae from differential geometry which we will
need when performing the later analysis.  We have as our principal object of study a smooth
immersion $f:M^n\rightarrow\R^{n+1}$ of an $n$-dimensional orientable compact hypersurface $M^n$,
and induced metric tensor with components
\[
g_{ij} = \eIP{\pD{}{x_i}f}{\pD{}{x_j}f},
\]
so that the pair $(M,g)$ is a Riemannian manifold.  In the above equation $\eIP{\cdot}{\cdot}$
denotes the regular Euclidean inner product, and $\pD{}{x_i}$ is the derivative in the direction of
the $i$-th tangent vector.  When convenient we frequently use the abbreviation $\partial_i =
\pD{}{x_i}$.

The Riemannian metric induces an inner product structure on all tensors,
which we define as the trace over pairs of indices with the metric:
\[
\IP{T^{i}_{jk}}{S^i_{jk}} = g_{is}g^{jr}g^{ku}T^i_{jk}S^s_{ru},\qquad \vn{T}^2 = \IP{T}{T},
\]
where repeated indices are summed over from $1$ to $n$.
The mean curvature $H$ is defined by
\[
H = g^{ij}A_{ij} = A_i^i,
\]
where the components $A_{ij}$ of the second fundamental form $A$ are given by
\begin{equation}
A_{ij} = -\eIP{\pD{{}^2}{x_i\partial x_j}f}{\nu}
       = \eIP{\pD{}{x_j}f}{\pD{}{x_i}\nu},
\label{PREPsff}
\end{equation}
where $\nu$ is the outer unit normal vector field on $M$.

The Christoffel symbols of the induced connection are determined by the metric,
\begin{equation}
\label{C3Echristoffelmetric}
\Gamma_{ij}^k = \frac{1}{2}g^{kl}
                \left(\pD{}{x_i}g_{jl} + \pD{}{x_j}g_{il} - \pD{}{x_l}g_{ij}\right),
\end{equation}
so that the covariant derivative on $M$ of a vector $X$ and of a covector $Y$ is
\begin{align*}
\nabla_jX^i &= \pD{}{x_j}X^i + \Gamma^i_{jk}X^k\text{, and}\\
\nabla_jY_i &= \pD{}{x_j}Y_i - \Gamma^k_{ij}Y_k
\end{align*}
respectively.

From the expression \eqref{PREPsff} and the smoothness of $f$ we can see that the second fundamental form is
symmetric; less obvious but equally important is the symmetry of the first covariant derivatives of
$A$, 
\[ \nabla_iA_{jk} = \nabla_jA_{ik} = \nabla_kA_{ij}, \]
commonly referred to as the Codazzi equation.

The fundamental relations between components of the Riemann curvature tensor, the Ricci tensor 
and scalar curvature are given by Gauss' equation
\begin{align*}
R_{ijkl} &= A_{ik}A_{jl} - A_{il}A_{jk},\intertext{with contractions}
g^{jl}R_{ijkl}
   = R_{ik} 
  &= HA_{ik} - A_i^jA_{jk}\text{, and}\\
g^{ik}R_{ik}
   = R
  &= H^2 - \vn{A}^2.
\end{align*}
We will need to interchange covariant derivatives; for vectors $X$ and covectors $Y$ we obtain
\begin{align*}
\nabla_{ij}X^h - \nabla_{ji}X^h &= R^h_{ijk}X^k = (A_{lj}A_{ik}-A_{lk}A_{ij})g^{hl}X^k,\\
\nabla_{ij}Y_k - \nabla_{ji}Y_k &= R_{ijkl}g^{lm}Y_m = (A_{lj}A_{ik}-A_{il}A_{jk})g^{lm}Y_m,
\end{align*}
where $\nabla_{i_1\ldots i_n} = \nabla_{i_1} \cdots \nabla_{i_n}$.
Further, we define $\nabla_{(n)}T$ to be the tensor with components $\nabla_{i_1\ldots
i_n}T_{j_1\ldots}^{k_1\ldots}$.
We also use for tensors $T$ and $S$ the notation $T*S$ to denote a
new tensor formed by summations of contractions of pairs of indices from $T$ and $S$ by the metric
$g$, with possible multiplication of each summation by a universal constant.  The resultant tensor
will have the same type as the other quantities in the equation it appears.  Keeping these in mind
we also denote polynomials in the iterated covariant derivatives of these terms by
\[
P_j^i(T) = \sum_{k_1+\ldots+k_j = i} c_{k_1\cdots k_j}\nabla_{(k_1)}T*\cdots*\nabla_{(k_j)}T,
\]
where the constant $c_{k_1\cdots k_j}\in\R$ is absolute.  As
is common for the $*$-notation, we slightly abuse this constant when certain subterms do not appear
in our $P$-style terms. For example
\begin{align*}
\vn{\nabla A}^2 
  &= \IP{\nabla A}{\nabla A}\\
  &= 1\cdot\left(\nabla_{(1)}A*\nabla_{(1)}A\right) + 0\cdot\left(A*\nabla_{(2)}A\right)\\
  &= P_2^2(A).
\end{align*}
This will occur throughout the paper without further comment.

The Laplacian we will use is the Laplace-Beltrami operator on $M^n$, with the components of $\Delta
T$ given by
\[
\Delta T^i_{jk} = g^{pq}\nabla_{pq}T^i_{jk} = \nabla^p\nabla_pT^i_{jk}.
\]
Using the Codazzi equation with the interchange of covariant derivative formula given above, we
obtain Simons' identity:
\begin{align}
\Delta A_{ij} &= \nabla_{ij}H + HA_{il}g^{lm}A_{mj} - \vn{A}^2A_{ij}\notag\\
              &= \nabla_{ij}H + HA_{i}^lA_{lj} - \vn{A}^2A_{ij},\notag\\
\intertext{or in $*$-notation}
\Delta A &= \nabla_{(2)}H + A*A*A.              \label{SimonsIdentity}
\end{align}
In the coming sections we will be concerned with calculating the evolution of the iterated covariant
derivatives of curvature quantities.  The following less precise interchange of covariant
derivatives formula (derived from the fundamental equations above) will be useful to keep in mind:
\[
\nabla_{ij}T = \nabla_{ji}T + P_2^0(A)*T.
\]
In most of our integral estimates, we will be including a function $\gamma:M\rightarrow\R$ in the
integrand.  Eventually, this will be specialised to a smooth cutoff function between concentric
geodesic balls on $M$.  For now however let us only assume that $\gamma = \tilde{\gamma}\circ f$,
where 
\begin{equation*}
0\le\tilde{\gamma}\le 1,\qquad\text{ and }\qquad 
\vn{\tilde{\gamma}}_{C^2(\R^{n+1})} \le c_{\tilde{\gamma}} < \infty.
\end{equation*}
Using the chain rule, this implies $D\gamma = (D\tilde{\gamma}\circ f)Df$ and then
$D^2\gamma = (D^2\tilde{\gamma}\circ f)(Df,Df) + (D\tilde{\gamma}\circ f)D^2f(\cdot,\cdot)$.
Using the expression
\eqref{C3Echristoffelmetric} for the Christoffel symbols to convert the computations above to
covariant derivatives, and the Weingarten relations to convert the derivatives of $\nu$ to factors
of the second fundamental form with the basis vectors $\partial_if$, we obtain the estimates
\begin{equation}
\label{e:gamma}
\vn{\nabla\gamma} \le c_{\gamma1},\qquad\text{ and }\qquad
\vn{\nabla_{(2)}\gamma} \le c_{\gamma2}(1+\vn{A}).
\end{equation}
At times we will use the set $[\gamma>c] = \{p\in M:\gamma(p)>c\}$ or the set
$[\gamma=c] = \{p\in M:\gamma(p)=c\}$ as the domain of integration.

\section{Integral estimates}

We now establish the fundamental integral estimates which allow us to exert control upon the
curvature and derivatives of curvature by controlling the concentration of the curvature.
Throughout this section we will need various Sobolev and interpolation inequalities.  These are
collected in the appendix for the convenience of the reader.

We begin with the following lemma, whose proof is straightforward, see \cite{huisken1999gee} for
example.

\begin{lem}\label{LemEV1}
For $f:M^n\times[0,T)\rightarrow\R^{n+1}$ evolving by $\partial_tf = F\nu$ the following equations
hold:
\label{l:curvevo}
\begin{align*}
  \pD{}{t}g_{ij} &= 2FA_{ij},\quad
  \pD{}{t}g^{ij} = -2FA^{ij},\quad
  \pD{}{t}d\mu = (HF)d\mu,\\
  \pD{}{t}\nu &= -\nabla F,\quad
  \pD{}{t}A_{ij} = -\nabla_{ij}F + FA_i^pA_{pj},\\
  \pD{}{t}H &= - \Delta F - F \vn{A}^2,\text{ and }\\
  \pD{}{t}A^o_{ij} &= -S^o(\nabla_{(2)}F) + F\big(A_i^pA_{pj}
                                          + \frac{1}{n}g_{ij}|A|^2-\frac{2}{n}HA_{ij}\big),
\end{align*}
where $S^o(T)$ denotes the tracefree part of a symmetric bilinear form $T$.
If $F = \Delta H + h$ then the following evolution equation additionally holds:
\[
  \pD{}{t}A_{ij} = -\Delta^2 A_{ij} + \vn{A}^2A_{ij} + (\Delta H - H + h)A_{ik}A^k_j.
\]
\end{lem}

\begin{lem}Let $f:M^n\times[0,T)\rightarrow\R^{n+1}$ be a constrained surface diffusion flow.  Then
the following equation holds:
\[
  \pD{}{t}\nabla_{(k)}A = -\Delta^2\nabla_{(k)}A + hP_2^k(A) + P_3^{k+2}(A).
\]
\end{lem}

\begin{cor}
Let $f:M^n\times[0,T)\rightarrow\R^{n+1}$ be a constrained surface diffusion flow.  Then the
following equation holds:
\[
  \pD{}{t}\vn{\nabla_{(k)}A}^2 = - 2\IP{\nabla_{(k)}A}{\nabla^p\Delta\nabla_p\nabla_{(k)}A}
                                 + [hP_2^k(A) + P_3^{k+2}(A)]*\nabla_{(k)}A.
\]
\label{S4cor2}
\end{cor}

Integration by parts gives us our most basic localised integral estimate.

\begin{cor}
\label{CorRE2}
Let $f:M^n\times[0,T)\rightarrow\R^{n+1}$ be a constrained surface diffusion flow, and $\gamma$ as
in \eqref{e:gamma}.  Then for any $s\ge0$,
\begin{align*}
\rD{}{t}\int_M \vn{\nabla_{(k)}A}^2\gamma^sd\mu + 2\int_M\vn{\nabla_{(k+2)}A}^2\gamma^sd\mu
   &=  \int_M \vn{\nabla_{(k)}A}^2(\partial_t\gamma^s)d\mu
\\ &\hskip-7cm
    + 2\int_M\IP{(\nabla\gamma^s)(\nabla_{(k)}A)}{\Delta\nabla_{(k+1)}A} d\mu
     - 2\int_M\IP{(\nabla\gamma^s)(\nabla_{(k+1)}A)}{\nabla_{(k+2)}A}d\mu 
\\ &\hskip-7cm
    + \int_M\gamma^s[(P_3^{k+2}(A)+hP_2^k(A))*\nabla_{(k)}A]d\mu.
\end{align*}
\end{cor}

Combining the above with standard integral estimates and interpolation inequalities as in
\cite{kuwert2002gfw} gives the following proposition.

\begin{prop}
\label{ASLTev1}
Let $f:M^n\times[0,T)\rightarrow\R^{n+1}$ be a simple constrained surface diffusion flow with 
$\gamma$ a cutoff function as in \eqref{e:gamma}.  Then for a fixed $\theta > 0$ and
$s\ge2k+4$,
\begin{align}
&\rD{}{t}\int_M \vn{\nabla_{(k)}A}^2\gamma^sd\mu
 + (2-\theta)\int_M \vn{\nabla_{(k+2)}A}^2\gamma^sd\mu\notag\\*
&\qquad \le (c+ch)\int_M \vn{A}^2\gamma^{s-4-2k}d\mu
            + ch\int_M\left(\nabla_{(k)}[A*A]*\nabla_{(k)}A\right)
                     \gamma^{s} d\mu
\notag\\*&\qquad\qquad\qquad
            + c\int_M \left([P_3^{k+2}(A)+P_5^k(A)]*\nabla_{(k)}A\right)\gamma^{s} d\mu,
\notag
\end{align}
where $c$ depends on $c_{\gamma1}$, $c_{\gamma2}$, $s$, $k$, $c_h([0,T))$, and $\theta$.
\end{prop}

We now use the above and specialised multiplicative Sobolev inequalities to demonstrate that small
concentration of curvature along the flow allows one to control the $L^2$ norm of first and second
derivatives of curvature.

\begin{prop} 
\label{ASLTkeyest1}
Let $n\in\{2,3\}$.  Suppose $f:M^n\times[0,T^*]\rightarrow\R^{n+1}$ is a simple constrained surface
diffusion flow and $\gamma$ a cutoff function as in \eqref{e:gamma}.  Then there is an $\epsilon_0$
depending on 
$c_{\gamma1}$, $c_{\gamma2}$, and $c_h([0,T^*])$ such that if
\begin{equation}
\label{ASLTkeyest1smallness}
\epsilon = \sup_{[0,T^*]}\int_{[\gamma>0]}\vn{A}^nd\mu\le\epsilon_0
\end{equation}
then for any $t\in[0,T^*]$ we have
\begin{equation}
  \begin{split}
&\int_{[\gamma=1]} \vn{A}^2 d\mu + \int_0^t\int_{[\gamma=1]} (\vn{\nabla_{(2)}A}^2 +
\vn{A}^2\vn{\nabla A}^2 + \vn{A}^6) d\mu d\tau \\
&\qquad\qquad\qquad \le \big(1+(n-2)t\big)\int_{[\gamma > 0]} \vn{A}^2 d\mu\Big|_{t=0}
 + c\big(t+(n-2)e^t\big)\epsilon^\frac{2}{n},
  \end{split}
\end{equation}
where $c$ depends on $c_{\gamma1}$, $c_{\gamma2}$, and $c_h([0,T^*])$.
\end{prop}
\begin{proof}
The idea of the proof is to integrate Proposition \ref{ASLTev1}, and then use the multiplicative
Sobolev inequality Lemma \ref{MS1lem}.  This will introduce a multiplicative factor of
$\vn{A}_{n,[\gamma>0]}$ in front of several integrals, which we can then absorb on the left.

Setting $k=0$, $s=4$ and $\theta=\frac{1}{2}$ in Proposition \ref{ASLTev1} we have
\begin{align}
&\rD{}{t}\int_M \vn{A}^2\gamma^4d\mu
 + \frac{3}{2}\int_M \vn{\nabla_{(2)}A}^2\gamma^4d\mu
\le (c+ch)\int_{[\gamma>0]} \vn{A}^2d\mu
\notag\\
&\qquad + ch\int_M\left([A*A]*A\right)
                     \gamma^{4} d\mu
            + c\int_M \left([P_3^{2}(A)+P_5^0(A)]*A\right)\gamma^{4} d\mu.
\label{ASLTkeyest1e1}
\end{align}
First we estimate the $P$-style terms:
\begin{align}
\int_M&\left([P_3^{2}(A)+P_5^0(A)]*A\right)\gamma^{4} d\mu
\notag\\*
&\le c\int_M\Big(\big[\vn{A}^2\cdot\vn{\nabla_{(2)}A}
                + \vn{\nabla A}^2\cdot\vn{A}+\vn{A}^5\big]\vn{A}\Big)\gamma^{4} d\mu
\notag\\&\le
     c\int_M\big[\vn{A}^3\cdot\vn{\nabla_{(2)}A}
                + \vn{\nabla A}^2\cdot\vn{A}^2+\vn{A}^6\big]\gamma^{4} d\mu
\notag\\&\le
     \theta\int_M \vn{\nabla_{(2)}A}^2\gamma^4d\mu
   +  c_\theta\int_M (\vn{A}^6 + \vn{\nabla A}^2\vn{A}^2)\gamma^{4} d\mu.
\notag
\intertext{
We use Lemma \ref{MS1lem} to estimate the second integral and obtain for $n=2$}
\int_M&\left([P_3^{2}(A)+P_5^0(A)]*A\right)\gamma^{4} d\mu
\notag\\*
&\le
     \theta\int_M \vn{\nabla_{(2)}A}^2\gamma^4d\mu
     + c_\theta\int_{[\gamma>0]}\vn{A}^2d\mu\int_M(\vn{\nabla_{(2)}A}^2 + \vn{A}^6)\gamma^4d\mu
\notag\\*&\qquad
     + c_\theta\Big( \int_{[\gamma>0]}\vn{A}^2d\mu \Big)^2,
\label{ASLTkeyest1eq1}
\intertext{and for $n=3$} 
\int_M&\left([P_3^{2}(A)+P_5^0(A)]*A\right)\gamma^{4} d\mu
\notag\\*
&\le
     \theta\int_M \vn{\nabla_{(2)}A}^2\gamma^4d\mu
     + c_\theta\vn{A}_{3,[\gamma>0]}^\frac{3}{2}\int_M(\vn{\nabla_{(2)}A}^2 + \vn{A}^6)\gamma^4d\mu
\notag\\*&\qquad
     + c_\theta\big(c_{\gamma1}\big)^3\big(\vn{A}^3_{3,[\gamma>0]} +
\vn{A}^\frac{9}{2}_{3,[\gamma>0]}\big).
\label{ASLTkeyest1eq2}
\end{align}
We add the integrals $\int_M \vn{A}^6 \gamma^4d\mu$ and $\int_M \vn{\nabla A}^2\vn{A}^2\gamma^4d\mu$
to \eqref{ASLTkeyest1e1} and obtain
\begin{align*}
&\rD{}{t}\int_M \vn{A}^2\gamma^4d\mu
 + \frac{3}{2}\int_M \big(\vn{\nabla_{(2)}A}^2+\vn{A}^2\vn{\nabla A}^2+\vn{A}^6\big)\gamma^4 d\mu
\\
&\qquad \le (c+ch)\int_{[\gamma>0]} \vn{A}^2d\mu
            + ch\int_M\left([A*A]*A\right)
                     \gamma^{4} d\mu
\\
&\qquad\qquad
     + c\int_M \big(\vn{A}^2\vn{\nabla A}^2+\vn{A}^6\big)\gamma^4 d\mu
     + c\int_M \left([P_3^{2}(A)+P_5^0(A)]*A\right)\gamma^{4} d\mu
\\
&\qquad \le c(1+h^2)\int_{[\gamma>0]} \vn{A}^2d\mu
     + c\int_M 
        \big(\vn{A}^3\vn{\nabla_{(2)}A}+\vn{A}^2\vn{\nabla A}^2+\vn{A}^6\big)\gamma^4 d\mu.
\intertext{
For $n=2$, we use the estimate \eqref{ASLTkeyest1eq1} above and obtain}
&\rD{}{t}\int_M \vn{A}^2\gamma^4d\mu
 + \frac{3}{2}\int_M \big(\vn{\nabla_{(2)}A}^2+\vn{A}^2\vn{\nabla A}^2+\vn{A}^6\big)\gamma^4 d\mu
\\
&\qquad \le c(1+h^2)\int_{[\gamma>0]} \vn{A}^2d\mu
     + \theta\int_M\vn{\nabla_{(2)}A}^2\gamma^4d\mu
\\
&\qquad\qquad
     + c_\theta\int_{[\gamma>0]}\vn{A}^2d\mu\int_M(\vn{\nabla_{(2)}A}^2 + \vn{A}^6)\gamma^4d\mu
     + c_\theta\Big( \int_{[\gamma>0]}\vn{A}^2d\mu \Big)^2.
\intertext{For $n=3$, we use instead \eqref{ASLTkeyest1eq2} to obtain}
&\rD{}{t}\int_M \vn{A}^2\gamma^4d\mu
 + \frac{3}{2}\int_M \big(\vn{\nabla_{(2)}A}^2+\vn{A}^2\vn{\nabla A}^2+\vn{A}^6\big)\gamma^4 d\mu
\\*
&\qquad \le c(1+h^2)\int_{[\gamma>0]} \vn{A}^2d\mu
     + \theta\int_M\vn{\nabla_{(2)}A}^2\gamma^4d\mu
\\*
&\qquad\qquad
     + c_\theta\vn{A}_{3,[\gamma>0]}^\frac{3}{2}\int_M(\vn{\nabla_{(2)}A}^2 + \vn{A}^6)\gamma^4d\mu
\\*&\qquad\qquad
     + c_\theta\big(c_{\gamma1}\big)^3\big(\vn{A}^3_{3,[\gamma>0]} + \vn{A}^\frac{9}{2}_{3,[\gamma>0]}\big).
\intertext{Absorbing, we obtain for $n=2$}
&\rD{}{t}\int_M \vn{A}^2\gamma^4d\mu
 + (3-2\theta-2c_\theta\epsilon_0)\frac{1}{2}
 \int_M \big(\vn{\nabla_{(2)}A}^2+\vn{A}^2\vn{\nabla A}^2+\vn{A}^6\big)\gamma^4d\mu
\\*
&\qquad\le c_\theta(1+\epsilon_0+\vn{h}^2_{\infty,[0,T^*]})\epsilon
\\*
&\qquad\le c_\theta\epsilon,
\intertext{and for $n=3$}
&\rD{}{t}\int_M \vn{A}^2\gamma^4d\mu
 + (3-2\theta-2c_\theta\sqrt{\epsilon_0})\frac{1}{2}
 \int_M \big(\vn{\nabla_{(2)}A}^2+\vn{A}^2\vn{\nabla A}^2+\vn{A}^6\big)\gamma^4d\mu
\\*
&\qquad\le c_\theta\big(1 + \vn{h}^2_{\infty,[0,T^*]}\big)\int_{[\gamma>0]}\vn{A}^2d\mu
             + c_\theta\big(\epsilon_0^\frac{1}{3}+\epsilon_0^\frac{5}{6}\big)\epsilon^\frac{2}{3}.
\end{align*}
For $\theta$, $\epsilon_0$ small enough we have
\begin{align*}
  \rD{}{t}\int_{M} \vn{A}^2 \gamma^4d\mu 
&+ \int_{M}
             \big(\vn{\nabla_{(2)}A}^2 + \vn{A}^2\vn{\nabla A}^2 + \vn{A}^6
             \big)\gamma^4 d\mu
\\
&\le c\epsilon^\frac{2}{n} + c(n-2)\int_{[\gamma>0]}\vn{A}^2d\mu,
\end{align*}
with 
$c$ depending on $\epsilon_0$, $c_h([0,t^*])$, $c_{\gamma1}$, and $c_{\gamma2}$.
Integrating, we have for $n=2$
\begin{align*}
\int_{[\gamma=1]} \vn{A}^2 \gamma^4d\mu\ +
 &\int_0^t\int_{[\gamma=1]} (\vn{\nabla_{(2)}A}^2 +
\vn{A}^2\vn{\nabla A}^2 + \vn{A}^6) d\mu d\tau\\*
&\le \int_{[\gamma>0]}\vn{A}^2d\mu\bigg|_{t=0} + c\epsilon t,
\end{align*}
where we used the fact $[\gamma=1]\subset[\gamma>0]$ and $0\le\gamma\le1$.
For $n=3$ we use a covering argument and Gronwall's inequality after integrating to obtain
\begin{align*}
\int_{[\gamma=1]}& \vn{A}^2 \gamma^4d\mu\ +
 \int_0^t\int_{[\gamma=1]} (\vn{\nabla_{(2)}A}^2 +
\vn{A}^2\vn{\nabla A}^2 + \vn{A}^6) d\mu d\tau\\*
&\le \int_{[\gamma>0]}\vn{A}^2d\mu\bigg|_{t=0} + c\epsilon^\frac{2}{3}t
   + c\int_0^t\Big(\int_{[\gamma>0]}\vn{A}^2d\mu\bigg|_{t=0} + c\epsilon^\frac{2}{3}\tau\Big)
              e^{\int_\tau^t c d\nu}d\tau
\\
&= (1+ct)\int_{[\gamma>0]}\vn{A}^2d\mu\bigg|_{t=0} + c\epsilon^\frac{2}{3}t
   + c\epsilon^\frac{2}{3}\int_0^t\tau e^{c(t-\tau)}d\tau
\\
&\le (1+ct)\int_{[\gamma>0]}\vn{A}^2d\mu\bigg|_{t=0}
   + c(t+e^t)\epsilon^\frac{2}{3}.
\end{align*}
This finishes the proof.
\end{proof}

We now move on to obtaining estimates for the higher derivatives of curvature in $L^\infty$.
The first issue is in dealing with the $P$-style terms from Proposition \ref{ASLTev1}.  These are
easily interpolated as in \cite{kuwert2002gfw} with the extra terms involving the constraint
function presenting little difficulty.

\begin{prop}
Suppose $f:M^n\times[0,T]\rightarrow\R^{n+1}$ is a constrained surface diffusion flow and
$\gamma$ a cutoff function as in \eqref{e:gamma}.  Then, for $s \ge 2k+4$ the following estimate
holds:
\begin{equation}
\label{ASLTintest}
  \begin{split}
&\rD{}{t}\int_{M} \vn{\nabla_{(k)}A}^2\gamma^s d\mu
 + \int_{M} \vn{\nabla_{(k+2)}A}^2 \gamma^s d\mu \\
&\qquad\qquad\qquad 
  \le c\vn{A}^4_{\infty,[\gamma>0]}\int_M\vn{\nabla_{(k)}A}^2\gamma^sd\mu
       + c\vn{A}^2_{2,[\gamma>0]}(1+\vn{A}^4_{\infty,[\gamma>0]}) \\
&\qquad\qquad\qquad\qquad
       +ch\left(h^\frac{1}{3}\int_M\vn{\nabla_{(k)}A}^2\gamma^sd\mu
              + (1+h^\frac{1}{3})\vn{A}^2_{2,[\gamma>0]}\right).
  \end{split}
\end{equation}
\end{prop}

We now prove that controlling the concentration of curvature in a ball gives pointwise bounds on all
derivatives of curvature in that ball.

\begin{prop} 
\label{ASLTkeyest2}
Let $n\in\{2,3\}$.
Suppose $f:M^n\times[0,T^*]\rightarrow\R^{n+1}$ is a simple constrained surface diffusion flow
and $\gamma$ is as in \eqref{e:gamma}.
Then there is an $\epsilon_0$ depending on the constants in \eqref{e:gamma} and $c_h([0,T^*])$ such
that if
\begin{equation}
\label{ASLTkeyest2smallness}
\sup_{[0,T^*]}\int_{[\gamma>0]}\vn{A}^nd\mu\le\epsilon_0,
\end{equation}
we can conclude
\begin{equation}
\vn{\nabla_{(k)}A}^2_{\infty,[\gamma=1]}
 \le c
\end{equation}
where $c$ depends on $k$, $T^*$, $c_{\gamma1}$, $c_{\gamma2}$, $c_h([0,T^*])$, and $\alpha_0(k+2)$.
The latter is defined by
\[
\alpha_0(k) = \sum_{j=0}^k \vn{\nabla_{(j)}A}_{2,[\gamma>0]}\bigg|_{t=0}.
\]
\end{prop}
\begin{proof}
As before, the idea is to use our previous estimates and then integrate.  The $\epsilon_0$ which we
will use is exactly the same as that in Proposition \ref{ASLTkeyest1}.
We fix $\gamma$ and consider cutoff functions $\gamma_{\sigma,\tau}$ which will allow us to combine
our previous estimates.  Define for $0\le\sigma<\tau\le 1$ functions $\gamma_{\sigma,\tau} =
\psi_{\sigma,\tau}\circ\gamma$ satisfying $\gamma_{\sigma,\tau}=0$ for $\gamma\le\sigma$ and
$\gamma_{\sigma,\tau}=1$ for $\gamma\ge\tau$.  The function $\psi_{\sigma,\tau}$ is chosen such that
$\gamma_{\sigma,\tau}$ satisfies \eqref{e:gamma}, although with different constants.
Acceptable choices are
\[
 c_{\gamma_{\sigma,\tau}1} = \vn{\nabla\psi_{\sigma,\tau}}_\infty\cdot c_{\gamma1},
\text{ and }
c_{\gamma_{\sigma,\tau}2}
 = \max\{c_{\gamma1}^2\vn{\nabla_{(2)}\psi_{\sigma,\tau}}_\infty,
         c_{\gamma2}\vn{\nabla\psi_{\sigma,\tau}}_\infty\}.
\]
Using the cutoff function $\gamma_{0,\frac{1}{2}}$ instead of $\gamma$ in Proposition
\ref{ASLTkeyest1} gives
\begin{align}
\int_{[\gamma_{0,\frac{1}{2}}=1]}\vn{A}^2d\mu+
  \int_0^{T^*}\int_{[\gamma_{0,\frac{1}{2}}=1]}\vn{\nabla_{(2)}A}^2+\vn{A}^6d\mu d\tau
    &\le c\epsilon_0^\frac{2}{n}T^* + \vn{A}^2_{2,[\gamma>0]}\Big|_{t=0}
\notag
\intertext{which is for $n=2$}
\int_{[\gamma\ge\frac{1}{2}]}\vn{A}^2d\mu+
  \int_0^{T^*}\int_{[\gamma\ge\frac{1}{2}]}\vn{\nabla_{(2)}A}^2+\vn{A}^6d\mu d\tau
    &\le c(1+T^*)\epsilon_0
\label{ASLTkeyest2eq1}
\intertext{and for $n=3$}
\notag
\int_{[\gamma\ge\frac{1}{2}]}\vn{A}^2d\mu+
  \int_0^{T^*}\int_{[\gamma\ge\frac{1}{2}]}\vn{\nabla_{(2)}A}^2+\vn{A}^6d\mu d\tau
    &\le c(1+T^*)\big(\delta + \epsilon_0^\frac{2}{3}\big),
\end{align}
where $\delta = \vn{A}^2_{2,[\gamma>0]}\big|_{t=0}$.  Note that we do not need any smallness of
$\delta$.

Recall the multiplicative Sobolev inequality Proposition \ref{MS2prop}:
\begin{equation*}
  \vn{T}^4_{\infty,[\gamma=1]}
    \le c\vn{T}^{4-n}_{2,[\gamma>0]}\big( \vn{\nabla_{(2)}T}^n_{2,[\gamma>0]}
                                   + \vn{TA^2}^n_{2,[\gamma>0]}
                                   + \vn{T}^n_{2,[\gamma>0]}\big).
\tag{\ref{MS2}}
\end{equation*}
Using this with $\gamma_{\frac{1}{2},\frac{3}{4}}$ and \eqref{ASLTkeyest2eq1} above we
obtain for $n=2$
\begin{align}
\int_0^T\vn{A}^4_{\infty,[\gamma\ge\frac{3}{4}]}d\tau
  &\le c\epsilon_0(c\epsilon_0(1+T^*)+\epsilon_0T^*)
\notag\\
  &\le c\epsilon_0.
\label{ASLTkeyest2eq2}
\end{align}
For $n=3$ we similarly obtain
\begin{align}
\int_0^T\vn{A}^4_{\infty,[\gamma\ge\frac{3}{4}]}d\tau
  &\le \sqrt{c(1+T^*)\big(\delta + \epsilon_0^\frac{2}{3}\big)}
               \big[c(1+T^*)\big(\delta + \epsilon_0^\frac{2}{3}\big)\big]^\frac{3}{2}
\notag\\
  &\le c\big(\sqrt{\delta} + \epsilon_0^\frac{1}{3}\big),
\end{align}
where $c$ depends on $c_h([0,T^*])$, $c_{\gamma1}$, $c_{\gamma2}$, $T^*$, $n$, and $\epsilon_0$.

We now use \eqref{ASLTintest} with $\gamma_{\frac{3}{4},\frac{7}{8}}$.  Factorising, we have
\begin{align*}
\rD{}{t}\int_{M} \vn{\nabla_{(k)}A}^2\gamma_{\frac{3}{4},\frac{7}{8}}^s d\mu
&
  \le c\vn{A}^4_{\infty,[\gamma_{\frac{3}{4},\frac{7}{8}}\ge0]}
       \int_M\vn{\nabla_{(k)}A}^2\gamma_{\frac{3}{4},\frac{7}{8}}^sd\mu
\\
&\qquad
       + c\vn{A}^2_{2,[\gamma_{\frac{3}{4},\frac{7}{8}}\ge0]}
          \Big(1+h+\vn{A}^4_{\infty,[\gamma_{\frac{3}{4},\frac{7}{8}}\ge0]}\Big)
\\
&\qquad
       +ch^\frac{4}{3}\Big(\int_M\vn{\nabla_{(k)}A}^2\gamma_{\frac{3}{4},\frac{7}{8}}^sd\mu
              + \vn{A}^2_{2,[\gamma_{\frac{3}{4},\frac{7}{8}}\ge0]}\Big)
\\
&
  \le c\Big(\vn{A}^4_{\infty,[\gamma\ge\frac{3}{4}]}+h^\frac{4}{3}\Big)
       \int_M\vn{\nabla_{(k)}A}^2\gamma_{\frac{3}{4},\frac{7}{8}}^sd\mu
\\
&\qquad
       + c\vn{A}^2_{2,[\gamma\ge\frac{3}{4}]}
         \Big(1+\vn{A}^4_{\infty,[\gamma\ge\frac{3}{4}]}+h+h^{\frac{4}{3}}\Big).
\end{align*}
Integrating,
\begin{align}
&\int_{M} \vn{\nabla_{(k)}A}^2\gamma_{\frac{3}{4},\frac{7}{8}}^sd\mu 
-\int_{M} \vn{\nabla_{(k)}A}^2\gamma_{\frac{3}{4},\frac{7}{8}}^sd\mu\bigg|_{t=0} 
\notag\\
&\qquad\qquad\qquad 
  \le c\int_0^t\left[\Big(\vn{A}^4_{\infty,[\gamma\ge\frac{3}{4}]}+h^\frac{4}{3}\Big)
                     \int_M\vn{\nabla_{(k)}A}^2\gamma_{\frac{3}{4},\frac{7}{8}}^sd\mu\right]d\tau
\notag\\
&\qquad\qquad\qquad\qquad
       + c\int_0^t
          \left[\vn{A}^2_{2,[\gamma\ge\frac{3}{4}]}
                \Big(1+\vn{A}^4_{\infty,[\gamma\ge\frac{3}{4}]}+h+h^{\frac{4}{3}}\Big)
          \right]d\tau.
\label{ASLTkeyest2eq3}
\end{align}
Now from our earlier calculation \eqref{ASLTkeyest2eq2} we have
\[
  \int_0^t\Big(\vn{A}^4_{\infty,[\gamma\ge\frac{3}{4}]}+h^\frac{4}{3}\Big)d\tau
\le c,
\]
and, using our assumption \eqref{ASLTkeyest2smallness}
\[
       c\int_0^t
          \left[\vn{A}^2_{2,[\gamma\ge\frac{3}{4}]}
                \Big(1+\vn{A}^4_{\infty,[\gamma\ge\frac{3}{4}]}+h+h^{\frac{4}{3}}\Big)\right]d\tau
\le c.
\]
Also, we have
\[
\int_{M} \vn{\nabla_{(k)}A}^2\gamma_{\frac{3}{4},\frac{7}{8}}^sd\mu\bigg|_{t=0}
 \le c\alpha_0(k),
\]
where $\alpha_0$ is as in the statement of the proposition.
Therefore, equation \eqref{ASLTkeyest2eq3} is of the form
\[
  \alpha(t) \le \beta(t) + \int_c^t \lambda(\tau)\alpha(\tau)d\tau,
\]
where
\begin{align*}
\alpha(t) &= \int_M \vn{\nabla_{(k)}A}^2\gamma_{\frac{3}{4},\frac{7}{8}}^sd\mu,\\
\beta(t) &= \int_M \vn{\nabla_{(k)}A}^2\gamma_{\frac{3}{4},\frac{7}{8}}^sd\mu\bigg|_{t=0}
     + c\int_0^t
          \left[\vn{A}^2_{2,[\gamma\ge\frac{3}{4}]}
                \Big(1+\vn{A}^4_{\infty,[\gamma\ge\frac{3}{4}]}+h+h^{\frac{4}{3}}\Big)\right]d\tau,
\intertext{and}
\lambda(t) &= 
        \vn{A}^4_{\infty,[\gamma\ge\frac{3}{4}]}+h^{\frac{4}{3}}.
\end{align*}
Noting that $\beta$ and $\int \lambda d\tau$ are bounded by the constants shown above, we can invoke
Gronwall's inequality and conclude 
\begin{equation*}
  \int_{[\gamma\ge\frac{7}{8}]} \vn{\nabla_{(k)}A}^2d\mu
\le \beta(t) + \int_0^t \beta(\tau)\lambda(\tau)e^{\int_\tau^t\lambda(\nu)d\nu}d\tau
\le c\big(k,\alpha_0(k)\big).
\end{equation*}
Trivially, we also have
\begin{equation*}
  \int_{[\gamma\ge\frac{7}{8}]} \vn{\nabla_{(k+2)}A}^2d\mu
\le c\big(k+2,\alpha_0(k+2)\big).
\end{equation*}
Therefore using \eqref{MS2} with $\gamma_{\frac{7}{8},\frac{15}{16}}$, and taking into account the
$n=3$ statement of Lemma \ref{MS1lem} we can bound $\vn{A}_\infty$ on a smaller ball:
\[
\vn{A}^{4}_{\infty,[\gamma\ge\frac{15}{16}]}
\le c(0,\alpha_0(0))^\frac{4-n}{2}\Big(
                                  \big(c(2,\alpha_0(2))^\frac{n}{2}+
                                  \big(c(0,\alpha_0(0))^\frac{n}{2}
                                  \Big)
\le c.
\]
Finally, using \eqref{MS2} with $T=\nabla_{(k)}A$ and $\gamma=\gamma_{\frac{15}{16},1}$ we 
obtain
\begin{align*}
\vn{\nabla_{(k)}A}_{\infty,[\gamma=1]}^4
 &\le c\vn{\nabla_{(k)}A}^{4-n}_{2,[\gamma>\frac{15}{16}]}
       \Big( \vn{\nabla_{(k+2)}A}^n_{2,[\gamma>\frac{15}{16}]}\\
&\hskip+3.5cm
                                   + (\vn{A}^{2n}_{\infty,[\gamma>\frac{15}{16}]}+1)
                                     \vn{\nabla_{(k)}A}^n_{2,[\gamma>\frac{15}{16}]}
       \Big)\\
 &\le c\big(k,\alpha_0(k+2)\big).
\end{align*}
This completes the proof.
\end{proof}

\section{Proof of the Lifespan Theorem}

We begin by scaling $\tilde{f}(x,t) = \frac{1}{\rho}f(x, \rho^4t)$.
Note that $\vn{A}_n^n$ is scale invariant, and so we may assume $\rho=1$.
Note that $h$ may scale in a non-invariant fashion
but this introduces a single change in the constant $c_h$ only, and certainly a scaled simple $h$
(we only perform this rescaling once) remains simple.
We make the definition
\begin{equation}
\label{ASLTe:epsfuncdef}
\eta(t) = \sup_{x\in\R^3}\int_{f^{-1}(B_1(x))} \vn{A}^nd\mu.
\end{equation}
By covering $B_1$ with several translated copies of $B_{\frac{1}{2}}$ there is a constant
$c_{\eta}$ such that
\begin{equation}
\label{ASLTe:epscovered}
\eta(t) \le c_{\eta}\sup_{x\in\R^3}\int_{f^{-1}(B_{\frac{1}{2}}(x))} \vn{A}^nd\mu.
\end{equation}
Note that $c_{\eta} = 4^{n+1}$ is sufficient.

By short time existence we have that $f(M\times[0,t])$ is compact for $t<T$ and so the function
$\eta:[0,T)\rightarrow\R$ is continuous.  We now define
\begin{equation}
\label{ASLTe:struweparameter}
t^{(n)}_0
=
  \sup\{0\le t\le\min(T,\lambda_n) : \eta(\tau)\le 3c_{\eta}\sigma(n)
                                         \ \text{ for }\ 0\le\tau\le t\},
\end{equation}
where
\begin{equation*}
\sigma(n) =
\begin{cases}
\epsilon_0 &\text{ for $n=2$,}
\\
c_{P8}c^*\big(\delta + \epsilon_0^{2/3}\big) &\text{ for $n=3$.}
\end{cases}
\end{equation*}
with $\delta = \sup_{x\in\R^4}\vn{A}^2_{2,f^{-1}(B_1(x))}\big|_{t=0}$, $\lambda_n$ a
parameter to be specified later and
\[
c^* = c_{P8} + c_0c_\eta e^{c_P5/c_0c_\eta}.
\]
The constant $c_{P8}$ is the maximum of 1 and the constant from Proposition \ref{ASLTkeyest2}, and
$c_0$ is the maximum of all the constants on the right hand side of Proposition \ref{ASLTkeyest1}.
Note that the $\epsilon_0$ on the right hand side of the inequality is from equation
\eqref{ASLTlifespansmallness}.  Unlike earlier in Proposition \ref{ASLTkeyest2}, we require $\delta$
small as described in the statement of Theorem \ref{ASLTlifespan}.

The proof continues in three steps.  First, we show that it must be the case that $t_0^{(n)} =
\min(T,\lambda_n)$.  Second, we show that if $t_0^{(n)} = \lambda_n$, then we can conclude the
Theorem \ref{ASLTlifespan}.  Finally, we prove by contradiction that if $T \ne \infty$, then
$t_0^{(n)} \ne T$.  We label these steps as 
\begin{align}
\label{ASLTe:7}
t_0^{(n)} &= \min(T,\lambda_n),\\
\label{ASLTe:8}
t_0^{(n)} &= \lambda_n \Longrightarrow
\text{Theorem \ref{ASLTlifespan}},\\
\label{ASLTe:9}
T &\ne \infty \Longrightarrow
t_0^{(n)} \ne T.
\end{align}
The three statements \eqref{ASLTe:7}, \eqref{ASLTe:8}, \eqref{ASLTe:9} together imply Theorem
\ref{ASLTlifespan}.  We expand the sketch of the argument given above as follows: first notice that
by \eqref{ASLTe:7} $t_0^{(n)} = \lambda_n$ or $t_0^{(n)} = T$, and if $t_0^{(n)} = \lambda_n$ then
by \eqref{ASLTe:8} we have Theorem \ref{ASLTlifespan}.  Also notice that if $t_0^{(n)} = \infty$
then $T = \infty$ and Theorem \ref{ASLTlifespan} follows from estimate
\eqref{ASLTe:prop44pluscovering} below (used to prove statement \eqref{ASLTe:8}).  Therefore the
only remaining case where the Theorem \ref{ASLTlifespan} may fail to be true is when $t_0^{(n)} = T
< \infty$.  But this is impossible by statement \eqref{ASLTe:9}, so we are finished.

We now give the proof of the first step, statement \eqref{ASLTe:7}.
From the assumption \eqref{ASLTlifespansmallness}, 
\[
\eta(0)\le\epsilon_0
 <
\begin{cases}
3c_{\eta}\epsilon_0,&\text{ for $n=2$}
\\
3c_{P8}c_\eta c^*\big(\delta + \epsilon_0^{2/3}\big),&\text{ for $n=3$},
\end{cases}
\]
and therefore \eqref{ASLTe:struweparameter} implies $t^{(n)}_0 > 0$.  Assume for the sake of
contradiction that $t^{(n)}_0 < \min(T,\lambda_n)$.  Then from the definition
\eqref{ASLTe:struweparameter} of $t_0^{(n)}$ and the continuity of $\eta$ we have
\begin{equation}
\label{ASLTe:t0ltmin}
\eta\big(t_0^{(n)}\big) = 
\begin{cases}
3c_{\eta}\epsilon_0,&\text{ for $n=2$}
\\
3c_{P8}c_\eta c^*\big(\delta + \epsilon_0^{2/3}\big),&\text{ for $n=3$,}
\end{cases}
\end{equation}
so long as $\epsilon_0 \le 1$ and $c_{P8}\ge1$.
Recall Proposition \ref{ASLTkeyest1}.  We will now set $\gamma$ to be a cutoff function as in
\eqref{e:gamma} such that
\[
\chi_{B_{\frac{1}{2}}(x)} \le \tilde{\gamma} \le \chi_{B_1(x)},
\]
for some $x\in f(M,t)$.
Choosing a small enough $\epsilon_0$ (by varying $\rho$ in \eqref{ASLTlifespansmallness}), definition
\eqref{ASLTe:struweparameter} implies that the smallness condition \eqref{ASLTkeyest1smallness} is
satisfied on $[0,t_0^{(n)})$.
Therefore we have satisfied all the requirements of Proposition \ref{ASLTkeyest1}, and so we
conclude
\begin{equation}
\label{ASLTe:useprop44}
  \begin{split}
  \int_{f^{-1}(B_{\frac{1}{2}}(x))}& \vn{A}^2 d\mu 
\\
         &\le
 \big(1+(n-2)t\big)\int_{f^{-1}(B_1(x))} \vn{A}^2 d\mu\Big|_{t=0}
 + c\big(t+(n-2)e^t\big)c_\eta\epsilon^\frac{2}{n}
\\
         &\le 
\begin{cases}
2\epsilon_0,
&\text{for $n=2$ and }\lambda_2 = \frac{1}{c_0c_{\eta}},
\\
2c_{P8}c^*\big(\delta + \epsilon_0^{2/3}\big)
&\text{for $n=3$ and }\lambda_3 = c_{P8}\frac{1}{c_0c_{\eta}},
\end{cases}
  \end{split}
\end{equation}
for all $t \in [0,t^*]$, where $t^* < t_0^{(n)}$.
Thus equation \eqref{ASLTe:useprop44} above is true for all $t \in \big[0,t_0^{(n)}\big)$.  We
combine this with \eqref{ASLTe:epscovered} to conclude
\begin{equation}
\label{ASLTe:prop44pluscovering}
\eta(t)
 \le c_{P8}^{n-2}c_{\eta} \sup_{x\in\R^3} \int_{f^{-1}(B_{\frac{1}{2}}(x))} \vn{A}^n d\mu 
 \le 
\begin{cases}
2c_{\eta}\epsilon_0,
&\text{ for $n=2$,}
\\
2c_{P8}c_{\eta}c^*\big(\delta+\epsilon_0^{2/3}\big),
&\text{ for $n=3$,}
\end{cases}
\end{equation}
where $0\le t < t_0^{(n)}$.
Since $\eta$ is continuous, we can let $t\rightarrow t_0^{(n)}$ and obtain a contradiction with
\eqref{ASLTe:t0ltmin}.  Therefore, with the choice of $\lambda_n$ in equation \eqref{ASLTe:useprop44}, the
assumption that $t_0^{(n)} < \min(T,\lambda_n)$ is incorrect.  Thus we have shown \eqref{ASLTe:7}.
We have also proved the second step \eqref{ASLTe:8}. Observe that if $t_0^{(n)} = \lambda_n$ then by
the definition \eqref{ASLTe:struweparameter} of $t_0^{(n)}$, 
\[
  T\ge\lambda_n,
\]
which is \eqref{ASLTlifespanTestimate}.  Also, \eqref{ASLTe:prop44pluscovering} implies
\eqref{ASLTlifespanCurvestimate}.  That is,
we have proved if $t_0^{(n)} = \lambda_n$, then the Lifespan Theorem holds, which is the second step
\eqref{ASLTe:8}.  It only remains to prove equation \eqref{ASLTe:9}.

We assume
\[
t_0^{(n)}=T\ne\infty;
\]
since if $T=\infty$ then \eqref{ASLTlifespanTestimate} holds automatically and again
\eqref{ASLTe:prop44pluscovering} implies \eqref{ASLTlifespanCurvestimate}.  Note also that we can
safely assume $T < \lambda_n$, since otherwise we can apply step two to conclude the Lifespan
Theorem.

Our strategy is to show that in this case the flow exists smoothly up to and including time $T$,
allowing us to extend the flow, thus contradicting the finite maximality of $T$.
Since $h$ is simple, it presents no difficulty, and for finite $T$, $h$ satisfies the requirements
of short time existence.
To show that the immersion $f(\cdot,T)$ satisfies the requirements of short time existence,  we
we use Proposition \ref{ASLTkeyest2} to obtain pointwise bounds for the higher derivatives of
curvature everywhere on $f(\cdot,T)$ and follow a standard proof such as that found in
\cite{kuwert2002gfw} or \cite{huisken1984fmc}.  Therefore we can extend the flow, contradicting the
maximality of $T$.
This establishes \eqref{ASLTe:9} and the theorem is proved.
\qed

\appendix
\section{Sobolev and Interpolation Inequalities}

Here we state the multiplicative Sobolev and interpolation inequalities we have used
in the paper.  We have generalised the inequalities in \cite{kuwert2002gfw} to the case of three
intrinsic dimensions.  Although the proofs are long and involved, they are straightforward and
standard and so we have omitted them, referring the reader to the appendix in \cite{kuwert2002gfw}
or \cite{mythesis} instead.

\begin{lem}
\label{MS1lem}
Let $\gamma$ be as in \eqref{e:gamma}.  Then for an immersed surface $f:M^2\rightarrow\R^{3}$ we
have
  \begin{align*}
  \int_M&\vn{A}^6\gamma^sd\mu + \int_M\vn{A}^2\vn{\nabla A}^2\gamma^sd\mu
\\
   &\le  c\int_{[\gamma>0]}\vn{A}^2d\mu\int_M(\vn{\nabla_{(2)}A}^2 + \vn{A}^6)\gamma^sd\mu
   + c(c_{\gamma1})^4\Big( \int_{[\gamma>0]}\vn{A}^2d\mu \Big)^2,
  \end{align*}
and for an immersion $f:M^3\rightarrow\R^4$,
  \begin{align*}
\int_M&\vn{A}^6\gamma^sd\mu + \int_M\vn{A}^2\vn{\nabla A}^2\gamma^sd\mu
\\
 &\le \theta\int_M \vn{\nabla_{(2)}A}^2\gamma^sd\mu
    +  c\vn{A}_{3,[\gamma>0]}^\frac{3}{2}
     \int_M \big(\vn{\nabla_{(2)} A}^2 + \vn{A}^6\big) \gamma^s d\mu
\\
&\qquad
    + c(c_{\gamma1})^3\big(\vn{A}^3_{3,[\gamma>0]} + \vn{A}^\frac{9}{2}_{3,[\gamma>0]}\big),
  \end{align*}
where $\theta\in(0,\infty)$ and $c$ is an absolute constant depending on $s$ and $\theta$.
\end{lem}

\begin{prop}
\label{MS2prop}
 Let $n\in\{2,3\}$.  Then for any tensor $T$
and $\gamma$ as in \eqref{e:gamma},
\begin{equation*}
  \vn{T}^4_{\infty,[\gamma=1]}
    \le c\vn{T}^{4-n}_{2,[\gamma>0]}\big( \vn{\nabla_{(2)}T}^n_{2,[\gamma>0]}
                                   + \vn{TA^2}^n_{2,[\gamma>0]}
                                   + \vn{T}^n_{2,[\gamma>0]}\big),
\label{MS2}
\end{equation*}
where $c$ depends on $c_{\gamma1}$, and $n$.
Assume $T=A$.
Then there exists an $\epsilon_0$ depending on $c_{\gamma1}$, $c_{\gamma2}$, and $n$ such that if 
\[
  \vn{A}^n_{n,[\gamma>0]} \le \epsilon_0
\]
we have
\begin{equation*}
  \vn{A}^{8n-12}_{\infty,[\gamma=1]}
    \le c\epsilon_0
         \big(\vn{\nabla_{(2)}A}^{2n^2-3n}_{2,[\gamma>0]}
           + \epsilon_0\big),
\label{MS2secondstatement}
\end{equation*}
with $c$ depending on $c_{\gamma1}$, $c_{\gamma2}$, $n$, and $\epsilon_0$.
\end{prop}

\bibliographystyle{plain}
\bibliography{mbib}
\end{document}